\documentclass[10pt]{article}
\usepackage{amsmath}
\usepackage{amsfonts}
\usepackage{amssymb}
\usepackage{amsthm}
\usepackage{url}
\usepackage[all]{xy}
\usepackage{dsfont}
\usepackage{graphicx}
\usepackage{caption}
\usepackage{subcaption}
\usepackage{comment} 
\usepackage{stmaryrd}
\usepackage{hyperref}
\usepackage{todonotes}
\usepackage{color}
\usepackage{enumerate,tikz-cd}
\usepackage{fixltx2e}
\usepackage{MnSymbol}
\usepackage{comment}
\allowdisplaybreaks

\numberwithin{equation}{section}

\newtheorem{proposition}{Proposition}[section]
\newtheorem{lemma}[proposition]{Lemma}
\newtheorem{theorem}[proposition]{Theorem}
\newtheorem{corollary}[proposition]{Corollary}

\theoremstyle{definition}
\newtheorem{remark}[proposition]{Remark}
\newtheorem{definition}[proposition]{Definition}

\DeclareMathOperator{\GL}{GL}

\DeclareMathOperator{\Spec}{Spec}

\renewcommand{\epsilon}{\varepsilon}

\renewcommand{\phi}{\varphi}

\newcommand\mi{^{-1}}

\newcommand\na{{\mathrm{NA}}}
\newcommand\rea[1]{\textcolor{orange}{#1}}

\newcommand\bbc{\mathbb{C}}
\newcommand\bbd{\mathbb{D}}

\newcommand\bbp{\mathbb{P}}
\newcommand\bbq{\mathbb{Q}}
\newcommand\bbr{\mathbb{R}}

\newcommand\cC{\mathcal{C}}
\newcommand\cD{\mathcal{D}}

\newcommand\cH{\mathcal{H}}

\newcommand\cL{\mathcal{L}}
\newcommand\cM{\mathcal{M}}

\newcommand\cO{\mathcal{O}}

\newcommand\cX{\mathcal{X}}

\newcommand{\norm}[1][\cdot]{\|#1\|}

\newcommand\field{{\mathrm{k}}}
\newcommand\ring{{\mathrm{k}^\circ}}
\newcommand\reldim{\mathrm{rel.dim}}
\newcommand\bigbbp{{\underline{\bbp}}}

\title{A Hilbert--Mumford criterion for generalised Monge--Ampère equations.}
\author{Rémi Reboulet}

\begin{document}

\maketitle

\begin{abstract}
    We give a new numerical criterion in the spirit of GIT for existence of solutions to inverse Hessian equations, including in particular the J-equation. Our criterion is formulated in terms of stability of pairs in the sense of Paul. To that end, we build on previous work of the author with Dervan, and generalise a result of Zhang, proving isometry between generalised Chow line bundles and mixed Deligne pairings.
\end{abstract}

\tableofcontents

\section*{Introduction}

The resolution of the Calabi conjecture by Yau \cite{yau:calabi} marks the beginning of a long series of works on generalised Monge--Ampère-type equations on complex projective (or Kähler) manifolds. Beyond the Calabi problem, the next simplest example of a Monge--Ampère equation involves an additional fixed Kähler form $\chi$, and asks for a solution $\omega$ in a given Kähler class to
$$n\omega^{n-1}\wedge\chi=C\omega^n,$$
with $C$ an adequate topological constant ensuring the equation to be numerically well-defined. This is known as the \textit{J-equation}, appearing already in \cite{don:momentmaps,xxchen:imrn} and which has first been solved in full generality by Gao Chen \cite{gaochen}. Interestingly, existence of a solution turns out to be equivalent to a numerical criterion generalising the classical result of Nakai and Moishezon for ample line bundles, suggesting that solvability of generalised Monge--Ampère equations is closely related new and interesting positivity conditions.

Since then (and even before Gao Chen's result), various works have focused on more general equations of mixed Monge--Ampère type \cite{colsze:jflow,song:nakai,datarpingali,fangma}, as well as on the deformed Hermitian Yang--Mills equation \cite{collinsyau,gaochen,chulee,chulee2,chulee3}. It is furthermore speculated that solving equations of this type in higher rank might help in solving the Hodge conjecture, see \cite{gaochen:hodge}.

We will take our interest here to a class of equations occasionally referred to as \textit{inverse $\sigma_k$-equations}, introduced by Fang--Lai--Ma \cite{fanglaima}. We will use algebraic techniques and therefore restrict to the projective setting, assuming $\chi$ and the sought solution $\omega$ lie in the Chern class of ample line bundles $H$ and $L$. The equation then takes the general form
\begin{equation}\label{eq:eqn2}
    \sum_{i=1}^n c_i \chi^i\wedge\omega^{n-i}=C\omega^n
\end{equation}
with $c_i>0$ real constants (for us, nonnegative integers), and $C$ again a cohomological constant fully determined by the $c_i$'s and the intersection numbers $(L^i\cdot H^{n-i})$. A version of Gao Chen's result in that case has been proven by Datar--Pingali \cite{datarpingali} and Song \cite{song:nakai}, where existence is again determined by a numerical criterion of Nakai--Moishezon type. Some earlier analytic developments related to this class of equations, that we will use in the present article, are also due to Collins--Szekelyhidi \cite{colsze:jflow}.

Our purpose here is to give an alternative criterion for solvability of the inverse $\sigma_k$-equations on a complex projective manifold, in the spirit of geometric invariant theory. We will obtain a criterion of Hilbert--Mumford type, using the theory of stability of pairs developed by Paul \cite{paul:hyperdiscriminants}. Our approach has the following benefits:
\begin{enumerate}
    \item it clearly emphasises the algebraic structure underlying functionals detecting solutions to mixed Monge--Ampère equations, much as Paul's fundamental work \cite{paul:hyperdiscriminants} did for the cscK equation;
    \item any topological result that applies to stability of pairs will yield results either for solvability of such equations in families, or for topological properties of classes satisfying the aforementioned numerical criteria;
    \item we obtain a criterion close in spirit to geometric invariant theory, testing against a very large class of objects called \textit{arcs}, suggesting much flexibility in constructing and checking counterexamples;
    \item finally, our results, provided some structural properties are known to hold, are likely to apply to a wide class of equations involving higher-rank classes, as discussed in the very last section.
\end{enumerate}

To state the main result, we introduce some terminology. First, given a very ample line bundle $L$ on a compact projective manifold $X$, it is a classical fact that one can associate to it a hypersurface in some projective space, the \textit{Chow hypersurface}, defined uniquely up to scaling by the \textit{Chow form}. The restriction of the hyperplane bundle to this hypersurface is called the \textit{Chow line}. A result of Zhang \cite{zhang:height} shows that the Chow line, endowed with the Fubini--Study metric, is \textit{isometric} to a line bundle called the \textit{Deligne pairing} $\langle L,\dots,L\rangle_X$ with its Deligne metric. The general Deligne pairing construction refines the intersection pairing, yielding a canonical line bundle in any intersection class $(L_0,\dots,L_n)$, possibly over a base that is not a point. This idea is initially due to Deligne \cite{deligne:detcoh} and has been developed by Elkik both in a purely algebraic \cite{elkik:1} and metrised \cite{elkik:2} setting. Many functionals in Kähler geometry, including the ones of interest to us, may be written as Deligne pairing metrics \cite{phongrosssturm,bhj:asymptotics}, and thus generalising the result of Zhang will allow one to reduce the study of such functionals to the study of (differences of) norms of Chow-type forms.

Paul \cite{paul:hyperdiscriminants} indeed discovered, without relying on the Deligne pairing machinery (although in our case, its use will be convenient due to the general structure of our equations being nicely encoded by this formalism), that the norm of the Chow form, together with the norm of another similar vector called the \textit{hyperdiscriminant}, allows one to recover the Mabuchi functional, whose coercivity is equivalent to the existence of a cscK metric \cite{chencheng:iiiexistence}. Thus, in the \textit{finite-dimensional setting} where one only looks at Fubini--Study on a very ample power $mL$, this functional is discovered to be purely algebraic, and its coercivity needs only be tested on one-parameter subgroups of $SL(H^0(X,mL))$ acting on the Chow form and the hyperdiscriminant. This has led Paul to introduce a version of geometric invariant theory involving not just one but two different representations, called \textit{stability of pairs}.

However, a counterexample to the Hilbert--Mumford criterion in this setting has recently been discovered by Paul--Sun--Zhang \cite{paulsunzhang}: one-parameter subgroups do not suffice to check stability of pairs. In recent work with Dervan \cite{dr:3}, we showed that the Hilbert--Mumford criterion still holds if one checks more general \textit{arcs}, which are $\bbc((t))$-points (or meromorphic maps from the punctured disc into) the algebraic group. Acting on the embedding of $X$ by sections of $mL$, an arc induces a degeneration of $(X,mL)$ called a \textit{model}, which is simply the non $\bbc^*$-equivariant version of the notion of a test configuration introduced by Donaldson \cite{donaldson:scalar}. Donaldson in a series of works \cite{donaldson:bstab,donaldson:bstab2,donaldson:bstab3}, as well as Wang \cite{wang}, have suggested the potential importance of considering models in the cscK problem. We have indeed used those to make progress on the Yau--Tian--Donaldson conjecture and obtain new results about automorphism groups of K-stable varieties in \cite{dr:3}. A purpose of the present article is therefore to highlight the fact that such ideas also apply to mixed Monge--Ampère problems.

In our setting, there is also a functional $J_{\chi,\underline c}$, introduced by X. Chen \cite{xxchen:imrn} (for the J-equation) and Collins--Szekelyhidi \cite{colsze:jflow} (in general), which detects existence of solutions to \eqref{eq:eqn2}. We show that this functional, when restricted to each space $\cH_m$ of Fubini--Study metrics, is expressed as a difference of norms associated to Chow-type points $R_{\chi,\underline c}(X,m)$, Theorem \ref{thm:lognorm}. To that end, we prove the following:

\bigskip\noindent\textbf{Theorem A.} (Theorems \ref{thm:mixedisomorphism}, \ref{thm:isometry}.)
    Let $L_0,\dots,L_n$ be very ample line bundles on a complex projective manifold $X$. Let $h_i$ denote the Fubini--Study metric induced by each embedding. Let $C_{\underline L}$ be the hypersurface cut out by the resultant of the $L_i$, $R_{\underline L}(X)$, endowed with its Chow-type line bundle $\cC_{\underline L}$ and the canonical metric on it (see Section \ref{sec2} for more details). Up to a universal constant independent of $X$, there is an isometry
    $$(\langle L_0,\dots,L_n\rangle_X,\langle h_0,\dots,h_n\rangle_X)=(\cC_{\underline L},\norm).$$

For the proof of Theorem A, we follow roughly the lines of the argument of Kapadia \cite{kapadia,kapadia:thesis} and Zhang \cite{zhang:height}, who proved this result for $L_0=\dots=L_{n+1}=L$ (see a slight generalisation in \cite{qtian}). Theorem A, of independent interest, further allows us to deduce Theorem \ref{thm:lognorm}, which again in the case where all line bundles are the same has been proven in various ways in the literature \cite{zhang:height,paul:chow,paul:hyperdiscriminants,phongsturm:stability}, together with other similar results in the \textit{global} case in Arakelov geometry \cite{arakelov:i,arakelov:ii,arakelov:iii,arakelov:iv,arakelov:v,chen:moriwakiintersection}. Building on the results of \cite{dr:3} and \cite{colsze:jflow,sze:lejmi,datarpingali}, we deduce from Theorem A our main theorem:

\bigskip\noindent\textbf{Theorem B.} (Theorem \ref{thm:b}.)
    The following are equivalent:
    \begin{enumerate}
        \item there exists a solution to \eqref{eq:eqn2} for any choice of $\chi'\in[\chi]$;
        \item there exists $\varepsilon=1/k>0$ such that the pair $(R_{\chi,\underline c}(X,mL),R(X,mL))$ (appearing in Theorem \ref{thm:lognorm} is $\varepsilon$-stable in the sense of stability of pairs (\ref{thm:stabeq}) for all $m$.
    \end{enumerate}

Since (1) is known to be equivalent to Nakai--Moishezon criteria \cite{datarpingali,song:nakai}, this also gives a stability of pairs interpretation for those.

\paragraph{Organisation of the paper.} Section \ref{sect1} contains preliminaries on Deligne pairings and their metrics. In Section \ref{sec2}, we define the resultant and Chow line as discussed above, prove Theorem A. In Section \ref{sec3}, we first review inverse $\sigma_k$-equations and their associated functionals, as well as notions about stability of pairs and models from \cite{dr:3}. We then prove Theorem B, and discuss some possible generalisations of the main result.

\paragraph{Acknowledgements.} The author thanks Ruadhaí Dervan, Siarhei Finski, and Jacob Sturm for discussions and comments.

\section*{Conventions.}

We use \textit{additive} notation for tensor products of line bundles, whereby $kL-M$ denotes $L^{\otimes -1}\otimes M^{\otimes k}$. An exponent $L^k$ or $L^{(k)}$ will denote the line bundle appearing $k$ times in either an intersection number, or a Deligne pairing.

We use \textit{multiplicative} notation for metrics $h_L^{\otimes -1}\otimes h_M^{\otimes k}$. This includes Deligne pairing metrics $\langle h_0,\dots,h_n\rangle$, although we will also sometimes switch to additive notation, which will be signaled by the use of curvature forms $\langle \omega_0,\dots,\omega_n\rangle$. 

%, \textit{except for Deligne pairings metrics} where we prefer to use additive notation. For example, the metric in the left-hand side of the statement of Theorem \ref{thm:isometry} is $e^{-C}\langle \omega_0,\dots,\omega_n\rangle_X)^{\otimes D}$ in tensor product notation. This might seem slightly confusing, but we believe it makes notation less cumbersome and more in the spirit of the change of metric formul\ae.

\section{Preliminaries.}\label{sect1}

Throughout, we follow the exposition of Elkik \cite{elkik:1,elkik:2}, Eriksson--Freixas \cite{eriksson:freixas}, and Yuan--Zhang \cite{yuanzhang}.

\subsection{Deligne pairings.}

The Deligne pairing machinery works in a very broad setting, but for our purposes it will be sufficient to consider $\pi:X\to S$ a holomorphic submersion between compact (connected) complex projective manifolds. Let $n=\reldim(X/S)$, $m=\dim S$ and let $L_0,\dots,L_n$ be line bundles on $X$. We describe the \textit{generators and relations} construction of Deligne pairings, following Elkik \cite[II.2.4]{elkik:1} and Zhang \cite{zhang:height}.

\begin{definition}
    The \textit{Deligne pairing}
    $$\langle L_0,\dots,L_n\rangle_{X/S}$$
    has as sections expressions of the form $\langle s_0,\dots,s_n\rangle$ where $s_i\in H^0(X,L_i)$ and $\cap_{i=0}^n (s_i)=\emptyset$, subject to relations given by the following data:
    \begin{enumerate}
        \item some $i\in[0,\dots,n]$ such that $\cap_{j\neq i}(s_j)=\sum a_k D_k$ is flat over $S$;
        \item a rational function $f$ on $X$;
        \item an identification $\langle s_0,\dots,fs_i,\dots,s_n\rangle=\sum_k a_k N_{D_k/S}(f)\langle s_0,\dots,s_n\rangle$, where $N_{D_k/S}$ denotes the norm functor.
    \end{enumerate}
    From this we can construct inductively transition functions, which define a genuine line bundle.
\end{definition}

The following isomorphism properties of Deligne pairings are then proven e.g.\ in \cite[Sections II-IV]{elkik:1}.

\begin{theorem}\label{thm:propdel}
    By the equality symbol below, we mean "there exists an isomorphism of line bundles over $S$". Furthermore, all isomorphisms commute with each other.
    \begin{enumerate}
        \item \textit{(norm in relative dimension zero)} if $n=0$,
        $$\langle L_0\rangle_{X/S}=N_{X/S}(L_0);$$
        \item \textit{(multilinearity)} for $k\in\bbr$, and given $L_0'$ another line bundle on $X$,
        $$\langle kL_0+L_0',\dots,L_n\rangle_{X/S}=k\langle L_0,\dots,L_n\rangle_{X/S}+\langle L_0',\dots,L_n\rangle_{X/S};$$
        \item \textit{(symmetry)} given a permutation $\sigma$, 
        $$\langle L_{\sigma(0)},\dots,L_{\sigma(n)}\rangle_{X/S}=\langle L_0,\dots,L_n\rangle_{X/S};$$
        \item \textit{(restriction)} given a section $s_i$ of $L_i$ defining a relative effective divisor $Y$,
        $$[s_i]:\langle L_0, L_1,\dots,L_n\rangle_{X/S}=\langle L_1|_Y,\dots,L_{i-1}|_Y,L_{i+1}|_Y,\dots,L_n|_Y\rangle_{Y/S},$$
        and
        $$\langle \cO_X, L_1,\dots,L_n\rangle_{X/S}=\cO_S.$$
    \end{enumerate}
    In particular, if $L_i=\cO_X$ for some $i$, then $\langle L_0,\dots,L_n\rangle_{X/S}=\cO_X$.
\end{theorem}

We also have the following important properties.

\begin{theorem}[Projection formula, {\cite[Section IV]{elkik:1}}]\label{thm:proj_formula_line}
    Consider the following commutative diagram:
    \begin{center}
    \begin{tikzcd}
    Y \arrow[rd, "\pi_Y"] \arrow[d, "\nu"] &   \\
    X \arrow[r, "\pi_X"']                  & S
    \end{tikzcd}
    \end{center}
    each arrow a holomorphic submersion, and set nonnegative integers $p,q$ with $p+q+2=\reldim(Y/S)+1$, line bundles $L_0,\dots,L_p$ on $X$ (for whose pullbacks $\nu^*L_i$ we still write $L_i$), and line bundles $M_0,\dots,M_q$ on $Y$. We then have the following isomorphisms, which commute with those of the above theorem:
    \begin{enumerate}
        \item if $q=\reldim(Y/X)$, hence $p=\reldim(X/S)-1$,
        $$\langle L_0,\dots,L_{p},M_0,\dots,M_{q}\rangle_{Y/S}=\langle L_0,\dots,L_p,\langle M_0,\dots,M_q\rangle_{Y/X}\rangle_{X/S};$$
        \item if $q=\reldim(Y/X)-1$, hence $p=\reldim(X/S)$,
        $$\langle L_0,\dots,L_{p},M_0,\dots,M_{q}\rangle_{Y/S}=\left(\int_{Y/X}\prod_{i=0}^q c_1(M_i)\right)\langle L_0,\dots,L_p\rangle_{X/S};$$
        \item if $q\leq \reldim(Y/X)-2$,
        $$\langle L_0,\dots,L_{p},M_0,\dots,M_{q}\rangle_{Y/S}=\cO_S.$$
    \end{enumerate}
    Furthermore, these isomorphisms also commute with those obtained from further pullbacks, in the sense of \cite[Corollary 6.12]{eriksson:freixas}.
\end{theorem}

\iffalse
\begin{theorem}[Birational invariance]\label{thm:birational_invariance}
    Consider the following commutative diagram:
    \begin{center}
    \begin{tikzcd}
    Y \arrow[rd, "\pi_Y"] \arrow[d, "\nu"] &   \\
    X \arrow[r, "\pi_X"']                  & S
    \end{tikzcd}
    \end{center}
    where $\pi_X$ and $\pi_Y$ satisfy the conditions of Definition \ref{def:assumptions_general}, and where $\nu$ is a proper birational morphism. Let $n=\reldim(X/S)$. Consider $\pi_X$-ample line bundles $L_0,\dots,L_n$ on $X$. Then, we have an isomorphism
    $$\langle L_0,\dots,L_n\rangle_{X/S}=\langle \nu^*L_0,\dots,\nu^*L_n\rangle_{Y/S}$$
    which commutes with the multilinearity and symmetry isomorphisms, and the projection formula isomorphisms \cite[Corollary 6.14]{eriksson:freixas}; furthermore if $L_i=\cO_X(D)$ for $D$ a relative effective Cartier divisor $\nu^*D$ is relative effective Cartier, and \rea{(...)}, then it also commutes with restriction isomorphisms along $D$ and $D'$.
\end{theorem}
\fi

\subsection{Deligne metrics.}

Assume now the $L_0,\dots,L_n$ to be (say) $\pi$-very ample, and let $h_0,\dots,h_n$ be relatively Kähler (or semipositive) metrics on them, with associated $\omega_i\in c_1(L_i)$. Elkik \cite{elkik:2} shows that one may associate to them a semipositive metric $\langle h_0,\dots,h_n\rangle_{X/S}$ on the Deligne pairing $\langle L_0,\dots,L_n\rangle_{X/S}$. Working fibrewise, so that we may assume $S$ is a point, then given $s_i\in H^0(X,L_i)$ such that $\langle s_0,\dots,s_n\rangle$ is a Deligne section, and setting, for each $i$, $Z_i=\{s_i=0\}$, we have
\begin{equation}\label{eq_delignemetric}
    -\log|\langle s_0,\dots,s_n\rangle|_{\langle h_0,\dots,h_n\rangle_{X}}=-\sum_{i=0}^n\int_{Z_{i-1}}\log|s_i|_{h_i}\omega_{i+1}\wedge\dots\wedge\omega_n,
\end{equation}
see the exposition of \cite[Section 4.2.2]{yuanzhang}. Varying the point in the base, it turns out that the metric is continuous \cite[Theorem 4.2.3]{yuanzhang} (see also \cite{moriwaki:deligne}). The pairing of metrics is further uniquely determined by the fact that the following holds:
\begin{theorem}\label{thm:iso}
    The following isomorphisms extend to isometries (where we do not specify the Deligne metrics when they are considered to be obvious):
    \begin{enumerate}
        \item multilinearity and symmetry in Theorem \ref{thm:propdel}(2,3) (\cite[Théorème I.1.1(c)]{elkik:2});
        \item restriction in Theorem \ref{thm:propdel}(4), where the norm of the isomorphism $[s_i]$ is
        $$\log\|[s_i]\|=-\int_{X/S}\log\|s\|_{h_i}\omega_{0}\wedge\dots\wedge\omega_{i-1}\wedge\omega_{i+1}\wedge\dots\wedge\omega_{n}$$
        (\cite[Théorème I.1.1(d)]{elkik:2});
        \item the projection formul\ae\, in Theorem \ref{thm:proj_formula_line} (\cite[Théorème I.2.2]{elkik:2}).
    \end{enumerate}
\end{theorem}
We also have the \textit{change of metric formula}: given a smooth $\phi$, we have
\begin{equation}\label{eq:changemetric2}
    \langle h_0e^{-\phi},h_1,\dots,h_n\rangle_{X/S}=e^{-\int_{X/S} \phi \omega_1\wedge\dots\wedge\omega_n}\langle h_0,\dots,h_n\rangle_{X/S},
\end{equation}
or additively
\begin{equation}\label{eq:changemetric}
    \langle \omega_0+dd^c\phi,\omega_1,\dots,\omega_n\rangle_{X/S}-\langle \omega_0,\dots,\omega_n\rangle_{X/S}=\int_{X/S} \phi \omega_1\wedge\dots\wedge\omega_n
\end{equation}
as a metric on $\cO_S$. Using the multilinearity and symmetry isometries below, one can then inductively deduce a formula for general pairings
$$\langle \omega_0+dd^c{\phi_0},\dots,\omega_n+dd^c{\phi_n}\rangle_{X/S}-\langle \omega_0,\dots,\omega_n\rangle_{X/S}.$$
Corollarily, we obtain the following $C^0$ estimate:
\begin{lemma}
    There exists a uniform constant $C=C(n,\underline L)$ such that, for any $\omega_i,\omega_i+dd^c\phi_i\in c_1(L_i)$,
    $$|\langle \omega_0+dd^c\phi_0,\dots,\omega_n+dd^c\phi_n\rangle_X-\langle \omega_0,\dots,\omega_n\rangle_X|\leq C\sum_i \sup_X |\phi_i|.$$
\end{lemma}

We conclude with a remark that will be essential in our proof later on, which appears in \cite[Section 6]{kapadia}.

\begin{remark}[Polarisation of a hypersurface via Deligne pairings]\label{rem:hypersurface}
    Let $X$ be a projective variety of dimension $n$, and let $L$ be a very ample line bundle on $X$. Pick a section $s_0\in H^0(X,L)$, and consider the hypersurface $Z$ defined by the vanishing of $s_0$. Let $\cC$ be the restriction of $\cO_{\bbp H^0(X,L)}(1)$ to the line $[s_0]\in \bbp H^0(X,L)$. Then, there is an $Aut(X,L)$-equivariant isomorphism
    \begin{equation}\label{eq:hypersurface}
        (L^n)\cC=\langle L^{(n-1)}\rangle_Z - \langle L^{(n)}\rangle_X.
    \end{equation}
    Furthermore, if $h$ is a Hermitian metric on $L$, we define a norm on $H^0(X,L)$ by
    \begin{equation}\label{eq:norm2}
        (L^n)\log\|s\|^2_h=\int_X \log|s|^2_h \omega^n,
    \end{equation}
    which yields a Hermitian metric on $\cC$. With this metric, and the Deligne pairing metric on the right-hand side, \eqref{eq:hypersurface} is then an isometry.
\end{remark}

\section{A general Deligne pairing isometry.}\label{sec2}

For convenience, we prove the main results of this section and the next in the case where the base $S=\Spec \bbc=\{x\}$ is a point, which is all we will need in the present paper; but the proofs naturally extend to (say) flat proper morphisms $\pi:X\to S$ with $S$ a more general base.

\subsection{The mixed Chow line.}

Let $X$ be a compact (connected) complex projective manifold of dimension $n$, and let $\underline L=(L_0,\dots,L_n)$ be a $(n+1)$-uple of very ample line bundles on $X$. We first consider the following objects:
\begin{enumerate}
    \item $V_i:=H^0(X,L_i)$, with dimension $\dim(V_i)=:r_i+1$;
    \item for each $i$, the \textit{incidence variety} $I_i$ defined by the vanishing of the canonical section of $\cO_{\bbp V_i}(1)+\cO_{\bbp V_i^\vee}(1)$ on $\bbp V_i\times \bbp V_i^\vee$ (throughout, a fibre product without specified base is simply taken over the base point);
    \item the fibre product $\underline I:=\times_{i=0}^{n}I_i$ inside $\bigbbp\times\bigbbp^\vee:=\left(\times_{i=0}^{n} \bbp V_i\right)\times\left(\times_{i=0}^{n}\bbp V_i^\vee\right)$;
    \item for each $i$, the embedding $\iota_i:X\hookrightarrow \bbp V_i$ yielding an embedding $\underline \iota:X\hookrightarrow \bigbbp$;
    \item the fibre product $\Gamma_{\underline L}:=X\times_\bigbbp \underline I$, which is a closed subvariety of $X\times\bigbbp^\vee$ (hence of $\bigbbp\times \bigbbp^\vee$); \iffalse in the diagram
    \begin{center}\begin{tikzcd}
\underline I\times_\bigbbp X \arrow[d, "p"] \arrow[r, "q"] & \underline I \arrow[r, hook] \arrow[d, "\pi"] & \bigbbp\times_S\bigbbp^\vee \arrow[d, "\pi"] \\
X \arrow[r, "\underline \iota", hook]                      & \bigbbp \arrow[r, no head, Rightarrow]        & \bigbbp                                     
\end{tikzcd};
\end{center}\fi
    \item finally, $C_{\underline L}:=q_*\Gamma_{\underline L}$, where $q$ is the projection $X\times\bigbbp^\vee\to\bigbbp^\vee$.
\end{enumerate}

We call $C_{\underline L}$ the \textit{mixed Chow hypersurface}, which has dimension $(\sum_i r_i)-1$. By \cite[Proposition 1.6.2, Definition 1.6.3]{chen:moriwakiintersection}, this is a hypersurface of multi-degree $(\delta_0,\dots,\delta_n)$, where
$$\delta_i:=(c_1(L_0)\cdot \ldots\cdot c_1(L_{i-1})\cdot c_1(L_{i+1})\cdot\ldots\cdot c_1(L_n));$$
it further is defined uniquely up to multiplication by $\lambda\in\bbc^*$ by an element
$$R_{\underline L}(X)\in \otimes_{i=0}^n S^{\delta_i}(V_i^\vee),$$
hence by a genuine unique element $[R_{\underline L}(X)]$ in the projectivisation of the above vector space. We call any such $R_{\underline L}(X)$ a \textit{mixed Chow form} for $X$ with respect to $\underline L$, and $[R_{\underline L}(X)]$ the \textit{mixed Chow point}.

\begin{remark}[Action of the $SL(V_i)$]\label{rem:action}
    There is an action of $SL(V_i)$ on $V_i$, hence on the embedding of $X$ inside $\bbp V_i$, and on the above vector space. Given $\sigma_i\in SL(V_i)$, it is easy to see that $R_{\underline L}((\sigma_0,\dots,\sigma_n)\cdot X)=(\sigma_0,\dots,\sigma_n)R_{\underline L}(X)$. In our applications to the inverse $\sigma_k$-equations, we will consider pairings of the form $\langle L^{k+1},H^{n-k}\rangle_X$ where we will not act on the $H$-embeddings, while we will consider the diagonal action of $SL(H^0(X,L))^{k+1}$ on $X$ embedded by sections by $L$. We will simply denote this action by $\sigma\cdot X$, in which case we will also have $R_{\underline L}(\sigma\cdot X)=\sigma\cdot R_{\underline L}(X)$.
\end{remark}

\begin{remark}[Interpretation of the incidence variety]\label{rem:incidence}
    Consider the following diagram, in order to denote the arrows:
    \begin{center}
    \begin{tikzcd}
\Gamma_{\underline L}=X\times_{\bigbbp} \underline I \arrow[r, "\iota", hook] & X\times \bigbbp^\vee \arrow[d, "p", two heads] \arrow[r, "q", two heads] & \bigbbp^\vee \arrow[r, "q_i", two heads] & \bbp V_i^\vee \\
                                                                              & X                                                  &                               &              
    \end{tikzcd}
    \end{center}
    Let $s_i$ be the canonical section of $p^*L_i + q^*q_i^*\cO_{\bbp V_i^\vee}(1)$. Defining inductively $\Gamma_0=X\times \bigbbp^\vee$, $\Gamma_i=\Gamma_{i-1}\cap \{s_i=0\}$, we find that $\Gamma_n=\Gamma_{\underline L}$.
\end{remark}

\begin{definition}\label{def:mixedchow}
    We denote by $\cO(1)_i^\vee:=q^*q_i^*\cO_{\bbp V_i^\vee}(1)$ for simplicity of notation. We also set
    $$\underline \cL:=\left(\sum_i \delta_i\cO(1)_i^\vee\right).$$
    This is an ample line bundle on $\bigbbp^\vee$, and thus defines an embedding of $\bigbbp^\vee$ (hence of $C_{\underline L}$) into $\bbp H^0(\bigbbp^\vee,\underline \cL)$. The \textit{mixed Chow line} $\cC_{\underline L}$ is the restriction of $\cO_{\bbp H^0(\bigbbp^\vee,\underline \cL)}(1)$ to the mixed Chow point $[R_{\underline L}(X)]\in \bbp H^0(\bigbbp^\vee,\underline L)$.% $C_{\underline L}\subset \bigbbp^\vee\subset \bbp H^0(\bigbbp^\vee,\underline \cL)$. From this point of view, we may identify $R_{\underline L}(X)$ can be identified with a section of the mixed Chow line.
    %$$\cC_{\underline L}:=\left\langle \underline \cL^{(\sum_{i=0}^n r_i)}\right\rangle_{C_{\underline L}}$$
    %on the mixed Chow hypersurface.
\end{definition}

\subsection{Isomorphism.}
\begin{theorem}\label{thm:mixedisomorphism}
    There exists an isomorphism
    $$\langle L_0,\dots,L_n\rangle_{X}=\cC_{\underline L}$$
    which is equivariant under the action in Remark \ref{rem:action}.
\end{theorem}
\begin{proof}
    Throughout we keep the notation of Definition \ref{def:mixedchow} and the diagram in Remark \ref{rem:incidence}. We follow the argument of Zhang \cite{zhang:height} and Kapadia \cite{kapadia}. Consider the Deligne pairing
    \begin{equation}
        \cD:=\left\langle \underline \cL^{(\sum_{i=0}^n r_i)},p^*L_0 + \cO(1)_0^\vee,\dots,p^*L_n + \cO(1)_n^\vee\right\rangle_{X\times\bigbbp^\vee}.
    \end{equation}
    First, applying the induction formula for Deligne pairings to the vanishing locus of each canonical section $s_i$, whose common vanishing locus is definitionally $\Gamma_{\underline L}$ as in Remark \ref{rem:incidence},  we find
    $$\cD=\left\langle\underline \cL^{(\sum_{i=0}^n r_i)}\right\rangle_{\Gamma_{\underline L}},$$
    hence
    \begin{equation}\label{eq:toisom}
    \cD=\left\langle \underline \cL^{(\sum_{i=0}^n r_i)}\right\rangle_{C_{\underline L}}.
    \end{equation}
    Returning to the initial expression of $\cM$ and expanding we find:
    \begin{align}\label{eq:dexpansion}
        \cD&=\left\langle \underline \cL^{(\sum_{i=0}^n r_i)},p^*L_0,\dots,p^*L_n\right\rangle_{X\times\bigbbp^\vee}\\
        &+\sum_{i=0}^n \left\langle\underline \cL^{(\sum_{i=0}^n r_i)},p^*L_0,\dots,p^*L_{i-1},\cO(1)_i^\vee,p^*L_{i+1},\dots,p^*L_n\right\rangle_{X\times\bigbbp^\vee}\\
        &+\text{higher order terms in the pullback from }\bigbbp^\vee.
    \end{align}
    The higher order terms are all trivial by Theorem \ref{thm:proj_formula_line}(3). The first term is seen by Theorem \ref{thm:proj_formula_line}(2) to equal
    \begin{equation}
        D\langle L_0,\dots,L_n\rangle_{X}
    \end{equation}
    with $D=(\cL^{(\sum_{i=0}^n r_i)})$. Terms of the second type will likewise be of the form
    \begin{equation}
        \delta_i\left\langle \underline \cL^{(\sum_{i=0}^n r_i)},\cO(1)_i^\vee\right\rangle_{\bigbbp^\vee}
    \end{equation}
    which, all summed together, yields
    \begin{equation}\label{eq:28}
        \cD=\left\langle \underline \cL^{(\sum_{i=0}^n r_i)}\right\rangle_{C_{\underline L}}=D\langle L_0,\dots,L_n\rangle_{X}+\left\langle \underline \cL^{(\sum_{i=0}^n r_i)+1}\right\rangle_{\bigbbp^\vee},
    \end{equation}
    equivalently
    \begin{equation}\label{eq:end}
        D\langle L_0,\dots,L_n\rangle_{X}=\left\langle \underline \cL^{(\sum_{i=0}^n r_i)}\right\rangle_{C_{\underline L}}-\left\langle \underline \cL^{(\sum_{i=0}^n r_i)+1}\right\rangle_{\bigbbp^\vee}.
    \end{equation}
    We conclude using Remark \ref{rem:hypersurface} applied to (in the notations of the remark) $X=\bigbbp^\vee$, $L=\underline \cL$, $s$ the defining section of $\underline \cL$ for $C_{\underline L}$, and $Z=C_{\underline L}$.
\end{proof}

\subsection{Isometry.}

We denote by $h_i$ the Fubini--Study metric on $L_i$ induced by the embedding $X\hookrightarrow \bbp V_i$, giving a metric $p^*h_i$ on $p^*L_i$, and a Deligne metric $\langle h_0,\dots,h_n\rangle_X$ on the Deligne pairing $\langle L_0,\dots,L_n\rangle_X$. We also endow $\cO_{\bbp V_i^\vee}(1)$ with the Fubini--Study metric, that we denote by $h_i^\vee$. This yields a metric $q^*q_i^*h_i^\vee$ on $\cO(1)^\vee_i$, and
\begin{equation}\label{eq:metric}
h_{\underline L}:=\otimes_{i=0}^n (q^*q_i^*h_i^\vee)^{\otimes \delta_i}
\end{equation}
on $\underline \cL$. Finally, the mixed Chow line $\cC_{\underline L}$ is endowed with the metric $\norm_{h_{\underline L}}$ induced by $h_{\underline L}$ as in \eqref{eq:hypersurface}, using the notation of Remark \ref{rem:hypersurface}.

\begin{theorem}\label{thm:isometry}
    There is an isometry
    $$(\langle L_0,\dots,L_n\rangle_{X},(e^{-C}\langle h_0,\dots,h_n\rangle_X))=(\cC_{\underline L},\norm_{h_{\underline L}}),$$
    where $C$ is a constant independent of $X$.
\end{theorem}
\begin{proof}
As discussed in Remark \ref{rem:hypersurface}, there is an isometry $$\cC_{\underline L}=\left\langle \underline \cL^{(\sum_{i=0}^n r_i)}\right\rangle_{C_{\underline L}}-\left\langle \underline \cL^{(\sum_{i=0}^n r_i)+1}\right\rangle_{\bigbbp^\vee}.$$ It therefore suffices to show that there is also an isometry in \eqref{eq:end}, that is:
\begin{equation}
    \left\langle \underline \cL^{(\sum_{i=0}^n r_i)}\right\rangle_{C_{\underline L}}-\left\langle \underline \cL^{(\sum_{i=0}^n r_i)+1}\right\rangle_{\bigbbp^\vee}=\langle L_0,\dots,L_n\rangle.
\end{equation}
Beginning at \eqref{eq:dexpansion}, since the multilinearity, symmetry, and projection isomorphisms are isometries by Theorem \ref{thm:iso}, we know that \eqref{eq:end} is isometric as long as we do not replace $\cD$ in the expression \eqref{eq:28}. That is,
\begin{equation}
    \langle L_0,\dots,L_n\rangle_{X}=\cD-\left\langle \underline \cL^{(\sum_{i=0}^n r_i)+1}\right\rangle_{\bigbbp^\vee}
\end{equation}
is an isometry. Thus, we are only left to show that \eqref{eq:toisom} extends to an isometry, i.e.\ that 
\begin{equation}\label{eq:toisom2}
    \cD=\left\langle \underline \cL^{(\sum_{i=0}^n r_i)}\right\rangle_{C_{\underline L}}
\end{equation}
is isometric. This follows from an induction statement as in \cite[Section 3.2]{kapadia:thesis} and the restriction isometry formula. Namely, letting $\Gamma_i$ be as in Remark \ref{rem:incidence}, we can define the Deligne pairings
\begin{equation}
    \cD_i:=\left\langle \underline \cL^{(\sum_{i=0}^n r_i)},p^*L_i + \cO(1)_0^\vee,\dots,p^*L_n + \cO(1)_n^\vee\right\rangle_{\Gamma_i}
\end{equation}
endowed with the metrics
\begin{equation}
    H_i:=\left\langle h_{\underline L}^{(\sum_{i=0}^n r_i)},p^*h_i\otimes q^*q_i^*h_i^\vee,\dots,p^*h_n\otimes q^*q_n^*h_n^\vee\right\rangle_{\Gamma_i}.
\end{equation}
By repeated induction, the restriction formula Theorem \ref{thm:iso}(2) shows that there is an isometry for each $i$, 
$$(\cD_{i-1},H_{i-1})|_{\Gamma_i}=(\cD_i,H_i e^{-f_i})$$
%where
%$$f_i=\int_{\Gamma_{i-1}}\log|s_i|_{p^*h_i\otimes q^*q_i^*h_i^\vee}(p^*\omega_{i+1}\otimes q^*q_{i+1}^*\omega_{i+1}^\vee)\wedge\dots\wedge(p^*\omega_n\otimes q^*q_n^*\omega_n^\vee)\wedge \omega_{\underline L}^{(\sum_{i=0}^n)}$$
hence an isometry
$$(\cD,\Omega)=\left(\Gamma_{\underline L},\left\langle h_{\underline L}^{((\sum_{i=0}^n r_i))}\right\rangle_{\Gamma_{\underline L}}e^{-f}\right),$$
with
$$f=\sum_{i=0}^n f_i.$$
Arguing as in \cite[(3.16)-(3.25)]{kapadia:thesis} (see also \cite[p. 85]{zhang:height}, \cite{stoll}, \cite[Lemma 3.8.6]{chen:moriwakiintersection}), one may compute the $f_i$ to show $f$ to be a universal constant, proving the result.
\end{proof}

\section{Application to mixed Monge--Ampère equations.}\label{sec3}

\subsection{Mixed Monge--Ampère equations and Deligne pairings.}\label{sec31}

Throughout, we fix two ample line bundles $L$, $H$ on $X$. Let $\chi$ be a fixed Kähler form in $c_1(H)$. We consider the following mixed Monge--Ampère equation (or \textit{inverse $\sigma_k$-equation} as in \cite{colsze:jflow}) seeking a Kähler $\omega\in c_1(L)$ such that
\begin{equation}\label{eq:eqn2}
    \sum_{i=1}^n c_i \chi^i\wedge\omega^{n-i}=C\omega^n
\end{equation}
for some fixed nonnegative constants $c_i$ (which will be integers for us), and the necessary cohomological constant $C$ which is fully determined by $\underline c=(c_1,\dots,c_n)$. Setting $c_i=0$ for $i>1$ we get the usual J-equation
$$\chi\wedge\omega^{n-1}=C\omega^n.$$
There are various functionals related to the equation \eqref{eq:eqn2}. Given $\phi$ such that $\omega_\phi=\omega+dd^c\phi$ is a Kähler form in $c_1(L)$, and defining $V=(L^n)$, we set
\begin{align}
    E(\phi):=&\,\frac{1}{V(n+1)}\sum_{i=0}^n\int_X \phi \omega_\phi^i\wedge\omega_\phi^{n-i},\\
    J_{\chi,\underline c}(\phi)=&\,V\mi\sum_{j=1}^n \left(\frac{c_j}{n+1-j} \sum_{i=0}^{n-j}\phi\chi^j\wedge\omega_\phi^i\wedge\omega_0^{n-j-i}\right) - C\cdot E(\phi),\\
    J_{\chi}(\phi):=J_{\chi,(1,0,\dots,0)}=&\,\frac{1}{Vn} \sum_{i=0}^{n-1}\phi\chi\wedge\omega_\phi^i\wedge\omega_0^{n-1-i} - C\cdot E(\phi),\\
    J(\phi):=&\,V\mi\int_X \phi\omega_{0}^n - E(\phi).
\end{align}
Using the change of metric formul\ae, it turns out we may rewrite our functionals as Deligne pairings:
\begin{align}
    E(\phi)=&\,\frac{1}{V(n+1)}(\langle \omega_\phi^{(n+1)}\rangle_X-\langle \omega_0^{(n+1)}\rangle_X),\\
    J_{\chi,\underline c}(\phi)=&\,V\mi\sum_{j=1}^n \left(\frac{c_j}{n+1-j} (\langle\chi^{(j)},\omega_\phi^{(n+1-j)}\rangle_X-\langle\chi^{(j)},\omega_0^{(n+1-j)}\rangle_X)\right) - C\cdot E(\phi),\\
    J_{\chi}(\phi)=&\,\frac{1}{Vn} (\langle\chi,\omega_\phi^{(n)}\rangle_X-\langle\chi,\omega_0^{(n)}\rangle_X) - C\cdot E(\phi),\\
    J(\phi)=&\,V\mi(\langle \omega_\phi,\omega_0^{(n)}\rangle_X-\langle \omega_0^{(n+1)}\rangle_X)- E(\phi).
\end{align}

We now state two essential properties of $J_{\chi,\underline c}$ due to Collins--Szekelyhidi.
\begin{theorem}[{\cite[Theorem 5, Proposition 23]{colsze:jflow}}]\label{thm:colsze}
    The two following results hold.
    \begin{enumerate}
        \item If there exists a solution to \eqref{eq:eqn2}, then there exists $\varepsilon>0,\delta$ such that
        \begin{equation}\label{eq:properness}
            J_{\chi,\underline c}>\varepsilon J-\delta
        \end{equation}
        on $\cH$.
        \item If \eqref{eq:properness} holds for $J_{\chi,\underline c}$, then it holds for $J_{\chi',\underline c}$ where $\chi'$ is any Kähler form in $[\chi]$.
    \end{enumerate}
\end{theorem}

\begin{remark}
    The expression of $J_{\chi,\underline c}$ on \cite[p.19, before Proposition 23]{colsze:jflow}, where it is defined as a normalised primitive, is not the same as our explicit expression in terms of mixed Monge--Ampère measures and Deligne pairings, but we leave it to the interested reader to check that a simple (but tedious) integration by parts argument shows both to be equal.
\end{remark}

\begin{remark}\label{rem:reduction}The result (2) above, stating that coercivity is independent of the choice of a reference $\chi$ in its class, will be essential for our approach: indeed, we will need to reduce to the case where $\chi$ is a (rescaled) Fubini--Study metric induced by the embedding by sections of a very ample power of $H$.
\end{remark}

\subsection{Stability of pairs.}

The next two sections review results from \cite{paul:hyperdiscriminants} and \cite{dr:3}. We will be brief in our exposition, and more details are given therein. Let $V,W$ be two finite-dimensional representations of a complex linear algebraic group $G$. 

Let $\field=\bbc((t))$, and define an \textit{arc} in $G$ to be a $\field$-point $\rho\in G(\field)$ -- analytically, this corresponds to a meromorphic map at zero $\bbc^*\to G$ from a small punctured disc. In particular, a one-parameter subgroup of $G$ is an arc.

Given $(v,w)\in V\oplus W$, we can see such a pair as a point in the ground field extension $(V\oplus W)_\field$, so that if $\rho$ is an arc, then also $\rho\cdot(v,w)\in (V\oplus W)_\field$. Decomposing this new vector in a basis extended from $V\oplus W$ and evaluating each component according to the order of vanishing at zero, as in \cite[Definition 3.3]{dr:3}, one defines a \textit{weight} $\nu(\rho,[v,w])$. One then has that:
\begin{theorem}[{\cite[Corollary 3.18]{dr:3}}]\label{thm:sstabeq}
    Let $(v,w)\in V\oplus W$, with both $v$ and $w$ nonzero, and denote by $[v,w]$ the corresponding point in $\bbp(V\oplus W)$. The following are equivalent:
    \begin{enumerate}
        \item the pair $[v,w]$ is \textit{semistable}, which definitionally \cite{paul:hyperdiscriminants} means $$\overline{G\cdot[v,w]}\cap \bbp(0\oplus W)=\emptyset;$$
        \item the pair $[v,w]$ is \textit{numerically semistable}, i.e.\ for all arcs $\rho\in G(\field)$,
        $$\nu(\rho,[v,w])\geq 0;$$
        \item the pair $[v,w]$ is \textit{analytically semistable}, i.e.\ for a pair (hence any pair) of Hermitian norms on $V$ and $W$, there exists $\delta\in\bbr$ such that for all $g\in G$,
        $$\log\norm[g\cdot v]-\log\norm[g\cdot w]\geq -\delta.$$
    \end{enumerate}
\end{theorem}

We also state the following \textit{slope formula}, which will be useful to us later on:
\begin{theorem}[{\cite[Lemma 3.16]{dr:3}}]\label{thm:slope_pair}
    Let $\rho$ be an arc in $G$, seen as a meromorphic map $\bbd^*\ni z\mapsto \rho(z)\in G$. Then as $|z|\to 0$,
    $$\log\|\rho(z)\cdot v\|-\log\|\rho(z)\cdot w\|=\nu(\rho,[v,w])\log|z|\mi + O(1).$$
\end{theorem}

The notion of \textit{stability} of pairs is more subtle, and requires a third representation: embed $G$ into $\GL(q)$ for some $q$ large enough. This gives a $G$-action on $\bbc^{q^2}$. Let $e$ be the identity in $\GL(q)\subset \bbc^{q^2}$, and let $d$ be the degree of the representation in the sense of \cite[(2.6)]{paul:hyperdiscriminants}. To define the numerical criterion, we also need the so-called \textit{norm} of $\rho$ with respect to $v$:
\begin{equation}\label{eq:norm}
    \|(\rho,v)\|:=\nu(\rho,[v,e^{\otimes d}]).
\end{equation}
\begin{theorem}[{\cite[Corollary 3.22]{dr:3}}]\label{thm:stabeq}
        Let $k>0$ be a positive integer. It is equivalent for a pair $[v,w]$ to be:
        \begin{enumerate}
            \item \textit{stable}, i.e.\ the pair $[e^{\otimes d}\otimes v^{\otimes k},w^{\otimes k+1}]$ is semistable;
            \item \textit{numerically stable}, i.e.\ for all arcs $\rho\in G(\field)$,
            $$\nu(\rho,[v,w])\geq (k+1)\mi \|(\rho,v)\|;$$
            \item \textit{analytically stable}, i.e.\ for a triple (hence any triple) of Hermitian norms on $V$, $W$, and $\bbc^{q^2}$, there exists $\delta\in\bbr$ such that for all $g\in G$,
            $$\log \|g\cdot v\|-\log\|g\cdot w\|\geq (k+1)\mi (d\log\|g\|-\log\|g\cdot v\|)-\delta.$$
        \end{enumerate}
     When it is needed to specify the integer $k$, we will say it is $\varepsilon$-stable (resp. numerically stable, analytically stable) for $\varepsilon=k\mi$. 
\end{theorem}

\subsection{Models and non-Archimedean functionals.}

We continue reviewing results from \cite{dr:3}. We first fix an ample line bundle on $X$, and we define:
\begin{definition}
    A \textit{model} of $(X,L)$ is the data of:
    \begin{enumerate}
        \item a flat, projective morphism $\cX\to\mathrm{Spec}\,\ring$ of schemes, where $\ring=\bbc[[t]]$;
        \item a relatively ample ($\bbq$-)line bundle $\cL$ on $\cX$;
        \item an isomorphism of the base change $(\cX_\field,\cL_\field)$ with the base change $(X_\field,L_\field)$.
    \end{enumerate}
    We say it is of \textit{exponent $r$} if $r\cL$ is relatively very ample.
\end{definition}
Analytically, a model corresponds to a family over a small disc, with fibres away from zero isomorphic to $(X,L)$. Test configurations, as involved in K-stability \cite{donaldson:scalar} but also in the search for numerical criteria for mixed Monge--Ampère equations \cite{gaochen}, are special cases of models whose isomorphisms away from zero are induced by a $\bbc^*$-action. In particular, deformations to the normal cones of subvarieties of $X$ are examples of test configurations, hence of models.

We recall the following:
\begin{proposition}[{\cite[Proposition 2.6]{dr:3}}]\label{prop:arcsmodels}
    Models of exponent $r$ for $(X,L)$ are in bijection with equivalence classes of arcs in $G$, where two arcs $\rho,\rho'\in G(\field)$ are equivalent iff $\rho\circ\rho'$ yields a point of $G(\ring)$.
\end{proposition}
This suggests that the numerical criterion for semistability of pairs is closely related to a numerical criterion for models. This turns out to indeed be the case, as we will explain later. For the moment, let us introduce the algebraic (or \textit{non-Archimedean}) version of our functionals above:
\begin{definition}
    We keep the notation and setup of Section \ref{sec31}. Let $(\cX,\cL)$ be a model of $(X,L)$, with structure morphism $\pi:\cX\to \mathrm{Spec}\,\ring$. We also denote $\pi^*H$ by $H$, and $\pi^*L$ by $L$, and $V=(L^n)$. Finally, all intersection numbers are taken in the sense of \cite[Definition 2.11]{dr:3}. We define:
    \begin{align}
        E^\na(\cX,\cL)=&\,\frac{1}{V(n+1)}(\cL^{n+1});\\
        J_{H,\underline c}^\na(\cX,\cL)=&\,V\mi\sum_{j=1}^n\frac{c_j}{n+1-j}(H^j\cdot \cL^{n+1-j}) - C\cdot E^\na(\cX,\cL),\\
        J(\cX,\cL)=&\,V\mi(\cL\cdot L^n)- E^\na(\cX,\cL).
    \end{align}
\end{definition}

\begin{remark}
    The expressions for our functionals seem to closely be related to Deligne pairings. It turns out that this is indeed the case: the intersection numbers above in fact correspond to (differences of) Deligne pairings for non-Archimedean metrics on Berkovich spaces, see e.g.\ \cite[Remark 3.13]{bj:trivval}, \cite[Section 8.3]{boueri}, \cite[Section 4]{reb:3}.
\end{remark}

We may finally state the \textit{slope formul\ae}\, from \cite[Theorem 4.5]{dr:3}, in the form that we will use here. A version of such results for test configurations appears in \cite{bhj:asymptotics}. 

Let us first set the following notation. Let $h_m$ be the Fubini--Study metric on $mL$ induced by the embedding by sections of a very ample power $mL$, with curvature form $\omega_m$. Given $g\in G:=SL(H^0(X,mL))$, there exists $\phi_g$ (unique up to constants) such that $g^*\omega_m=\omega_m+dd^c\phi_g$. We write $\cH_m:=\{m\mi \phi_g,\,g\in G\}$, a subset of the space of Kähler potentials $\cH_L$, which we call the set of \textit{Fubini--Study potentials} in degree $m$.
\begin{theorem}\label{thm:slope}
    Let $m$ be such that $mL$ is very ample. Let $\rho$ be an arc in $G:=SL(H^0(X,mL))$, and let $(\cX_\rho,\cL_\rho)$ be the model of exponent $m$ it induces by Proposition \ref{prop:arcsmodels}. We see $\rho$ as a meromorphic map $\bbd^*\ni z\mapsto \rho(z)\in G$. Then, as $|z|\to 0$,
    \begin{align}
        E(\phi_{\rho(z)})=&\,E^\na(\cX_\rho,\cL_\rho)\log|z|\mi + O(1);\\
        J_{\chi,\underline c}(\phi_{\rho(z)})=&\,J_{H,\underline c}^\na(\cX_\rho,\cL_\rho)\log|z|\mi + O(1);\\
        J(\phi_{\rho(z)})=&\,J^\na(\cX_\rho,\cL_\rho)\log|z|\mi + O(1).
    \end{align}
\end{theorem}

\subsection{$J$-functionals as log-norm functionals.}

Throughout, by Theorem \ref{thm:colsze}(2), we may as in Remark \ref{rem:reduction} without loss of generality assume that $\chi$ is the Fubini--Study metric associated to the embedding of $X$ into $\bbp H^0(X,pH)$ for a $p$ such that $pH$ is very ample. For simplicity of notation we will therefore assume $H$ itself to be very ample. Let $m$ be such that $mL$ is very ample, and let $h$ be the induced Fubini--Study metric on $mL$ with Kähler form $\omega$. We consider the Deligne pairing
$$\langle (mL)^{(n+1-i)},H^{(i)}\rangle_X.$$
Let $C_i=C_{((mL)^{(n+1-i)},H^{(i)})}$ be the associated Chow hypersurface, $\cC_i$ the mixed Chow line, and $[R_i(X,m)]$ the mixed Chow point. Given $g\in SL(H^0(X,mL))$, we have two metrics on the line $\langle (mL)^{(n+1-i)},H^{(i)}\rangle_X$, 
$$\langle h^{(n+1-i)},h_\chi^{(i)}\rangle_X\text{ and }\langle (g^*h)^{(n+1-i)},h_\chi^{(i)}\rangle_X.$$
Through the isomorphism Theorem \ref{thm:mixedisomorphism} and Remark \ref{rem:hypersurface}, we can equivariantly identify a section of this Deligne pairing as a section $s$ of the hyperplane bundle over the mixed Chow point $[R_i(X,m)]$, which we will denote slightly abusively by $R_i(X,m)$. We compute the value of the difference of metrics at a such $s$:
\begin{equation}
    \log\frac{|R_i(X,m)|_{\langle (g^*h)^{(n+1-i)},h_\chi^{(i)}\rangle_X}}{|R_i(X,m)|_{\langle h^{(n+1-i)},h_\chi^{(i)}\rangle_X}}=\log\frac{\norm[g\cdot R_i(X,m)]_{h_i}}{\norm[R_{i}(X,m)]_{h_i}}
\end{equation}
using Theorem \ref{thm:isometry} and Remark \ref{rem:hypersurface}; where $\norm_{h_i}$ is the norm in \eqref{eq:norm2}, having denoted by $h_i$ the metric $h_{\underline L}=h_{(mL^{(n+1-i},H^{(i)})}$ in \eqref{eq:metric} (adequately rescaled to fit with the universal constant in Theorem \ref{thm:isometry}). We thus obtain:
\begin{theorem}\label{thm:lognorm}
    Let $R_{\chi,\underline c}(X,m):=\otimes_{i=1}^n R_i(X,m)^{\otimes c_i}\in \otimes_{i=1}^n H^0(C_i,\cC_i)$, and let $\norm_{\chi,\underline c}$ denote the tensor product norm of the $\norm_{\omega_i}$ as above. Let $R(X,m)$ be the Chow point of $(X,mL)$, and let us simply denote by $\norm$ the associated norm. Given $\phi_g\in \cH_m$, we then have that
    $$J_{\chi,\underline c}(\phi_g)=\log \frac{\norm[g\cdot R_{\chi,\underline c}(X,m)]_{\chi,\underline c}}{\norm[R_{\chi,\underline c}(X,m)]_{\chi,\underline c}}-C\log\frac{\norm[g\cdot R(X,m)]}{\norm[R(X,m)]}.$$
\end{theorem}

\subsection{Proof of the main result.}

As a corollary of Theorem \ref{thm:lognorm} and the above results on stability of pairs and slope formul\ae, we first have:
\begin{corollary}
    The following are equivalent:
    \begin{enumerate}
        \item $J_{\chi,\underline c}$ is coercive on $\cH_m$;
        \item the pair $(R_{\chi,\underline c}(X,m),R(X,m))$ is stable;
        \item $J^\na_{\chi,\underline c}(\cX,\cL)>\varepsilon J^\na(\cX,\cL)$ for all models of exponent $m$.
    \end{enumerate}
\end{corollary}
\begin{proof}
    (1)$\Leftrightarrow$(2) follows from rewriting $J_{\chi,\underline c}(\phi_g)$ using Theorem \ref{thm:lognorm}, and $J(\phi_g)$ using Tian's result \cite[Lemma 3.2]{tian:stabpairs}.

    (2)$\Leftrightarrow$(3), follows from Theorem \ref{thm:stabeq} and comparing the slope formul\ae\,Theorems \ref{thm:slope_pair} and \ref{thm:slope}.
\end{proof}

We now conclude.

\begin{theorem}\label{thm:b}
    The following are equivalent:
    \begin{enumerate}
        \item there exists a solution to \eqref{eq:eqn2};
        \item there exists $k>0$ such that, for all $m$, the pair $(R_{\chi,\underline c}(X,m),R(X,m))$ is $\varepsilon=k\mi$-stable.
    \end{enumerate}
\end{theorem}
\begin{proof}
    (1) implies that $J_{\chi,\underline c}$ is coercive by Theorem \ref{thm:colsze}(1). This implies that $J_{\chi,\underline c}$ is coercive on each $\cH_m$ (with uniform $\varepsilon$). By the above corollary, this gives (2).
    
    Assume now (2) to hold. By the above corollary again, this is equivalent to coercivity of $J^\na_{\chi,\underline c}$ on all models of exponent $m$ with uniform $\varepsilon$ -- equivalently, it is coercive on all models. Thus it is in particular coercive on all test configurations, hence on deformations to the normal cone. By Lejmi--Szekelyhidi \cite[Proposition 13]{sze:lejmi} and the discussion of Section 7 therein, this implies the (uniform version of the) numerical criterion of Datar--Pingali \cite[Theorem 1.1(3)]{datarpingali}, which by \cite[Theorem 1.1]{datarpingali} implies existence of a solution, that is: (1).
\end{proof}

\subsection{General structure and comments.}

The author wants to mention the possible generalisation of the results of this article to equations involving higher-rank classes. Namely, let $X\to S$ be (say) a holomorphic submersion, $E_i\to X$ be vector bundles of respective rank $r_i+1$, for all $i\leq n\leq d+1$, where $d$ is the relative dimension. Let $\bbp:=\bbp E_1\times_X \dots\times_X\bbp E_n$, and for each $i$ with projection $p_i:\bbp\to\bbp E_i$ set $L_i:=p_i^*\cO_{\bbp E_i}(1)$. Given $k_i$ positive integers with $\sum_i k_i=d+1$, Elkik \cite{elkik:1} defines a Deligne pairing for Segre classes as
$$\langle s_{k_1}(E_1),\dots,s_{k_n}(E_n)\rangle_{X/S}:=\langle (L_1)^{(r_1+k_1)},\dots,(L_n)^{(r_n+k_n)}\rangle_{\bbp/S}.$$
One then obtains pairings for Chern classes by linearity and their usual expression in terms of Segre classes \cite[Section 7.1.1]{eriksson:freixas}. At the very least, our isomorphism result should extend to this setting. Given an equation involving such classes, if there exists a functional $F$ whose coercivity is (say) equivalent to existence of solutions, and which may be written as a Deligne pairing in this higher-rank sense, then it is clear that the slope formula in our article also holds for $F$ (the definition of the non-Archimedean version of $F$ being similar). Which formulation of the isometry result is expected is less clear. Solving such problems is likely to help proving results related to Z-critical equations as in \cite{dervan:zcrit1,dervan:zcrit2,dervan:zcrit3}, higher Mabuchi functionals \cite{paul:higherenergies,westrich}, and Gao Chen's programme for the Hodge conjecture \cite{gaochen:hodge}.

\bibliographystyle{alpha}
\bibliography{bib}

\end{document}